\pdfoutput=1
\RequirePackage{ifpdf}
\ifpdf 
\documentclass[pdftex]{sigma}
\else
\documentclass{sigma}
\fi

\usepackage{bbm}
\def \To{\longrightarrow}

\def \Vec{\operatorname{Vec}}

\def \Z{\mathbbm{Z}}

\def \R{\mathcal{R}}
\def \r{\mathbbm{R}}

\def \k{\mathbbm{k}}
\def \1{\mathbf{1}}

\def \Id{\operatorname{Id}}

\numberwithin{equation}{section}

\newtheorem{Theorem}{Theorem}[section]
\newtheorem{Corollary}[Theorem]{Corollary}
\newtheorem{Proposition}[Theorem]{Proposition}
 { \theoremstyle{definition}
\newtheorem{Definition}[Theorem]{Definition} }

\begin{document}


\newcommand{\arXivNumber}{1510.04408}

\renewcommand{\PaperNumber}{004}

\FirstPageHeading

\ShortArticleName{Generalized Clif\/ford Algebras as Algebras in Suitable Symmetric Linear Gr-Categories}

\ArticleName{Generalized Clif\/ford Algebras as Algebras\\ in Suitable Symmetric Linear Gr-Categories}

\Author{Tao {CHENG}~$^{\dag\ddag}$, Hua-Lin {HUANG}~$^{\dag}$ and Yuping {YANG}~$^{\dag}$}

\AuthorNameForHeading{T.~Cheng, H.-L.~Huang and Y.~Yang}

\Address{$^{\dag}$~School of Mathematics, Shandong University, Jinan 250100, China}
\EmailD{\href{mailto:taocheng@sdnu.edu.cn}{taocheng@sdnu.edu.cn}, \href{mailto:hualin@sdu.edu.cn}{hualin@sdu.edu.cn},
\href{mailto:yupingyang@mail.sdu.edn.cn}{yupingyang@mail.sdu.edn.cn}}

\Address{$^{\ddag}$~School of Mathematical Science, Shandong Normal University, Jinan 250014, China}

\ArticleDates{Received October 22, 2015, in f\/inal form January 06, 2016; Published online January 12, 2016}

\Abstract{By viewing Clif\/ford algebras as algebras in some suitable symmetric Gr-cate\-go\-ries, Albuquerque and Majid were able to give a new derivation of some well known results about Clif\/ford algebras and to generalize them. Along the same line, Bulacu observed that Clif\/ford algebras are weak Hopf algebras in the aforementioned categories and obtained other interesting properties. The aim of this paper is to study generalized Clif\/ford algebras in a~similar manner and extend the results of Albuquerque, Majid and Bulacu to the generalized setting. In particular, by taking full advantage of the gauge transformations in symmetric linear Gr-categories, we derive the decomposition theorem and provide categorical weak Hopf structures for generalized Clif\/ford algebras in a~conceptual and simpler manner.}


\Classification{15A66; 18D10; 16S35}

\section{Introduction}

Clif\/ford algebras and their generalizations have played important roles in various branches of science and engineering. Though they have been studied in depth for many years, some new ideas still get involved and novel properties and applications are found. In recent years, with the pioneering work of Albuquerque and Majid \cite{am1, am2}, Clif\/ford algebras can be seen as algebras in some suitable symmetric Gr-categories. Many well known results about Clif\/ford algebras thus can be derived by some novel ideas arising from the theory of tensor categories and this viewpoint also helps to make natural generalization. Along the same line, Bulacu observed that Clif\/ford algebras are weak Hopf algebras in the aforementioned categories and obtained other interesting properties, see \cite{b1}.

The aim of this paper is to study generalized Clif\/ford algebras, written simply GCAs in the following,  in a similar manner and extend the results of Albuquerque, Majid~\cite{am1, am2} and Bula\-cu~\cite{b, b1} to the generalized setting. We provide an explicit formula of 2-cocycles to present GCAs as twisted group algebras. This helps us to put them inside suitable symmetric Gr-categories and apply the working philosophy of Albuquerque--Majid. In particular, by taking full advantage of the gauge transformations in symmetric linear Gr-categories, we derive the decomposition theorem and provide completely the categorical weak Hopf structures for GCAs in a conceptual and simpler manner.

The paper is organized as follows. In Section~\ref{section2}, we recall the notion of symmetric linear Gr-categories and some examples of algebras inside them. The presentation of GCAs as explicit twisted group algebras is given in Section~\ref{section3}. The well known decomposition theorem of GCAs are also derived conceptually and in a very simple manner with a help of suitable gauge transformations. Section~\ref{section4} is devoted to the categorical weak Hopf structures of GCAs. The complete determination of the weak Hopf structures of GCAs in some suitable symmetric Gr-categories is translated to a much simpler situation via gauge transformations. This simplif\/ies, improves and generalizes the main results of~\cite{b, b1}.

Throughout the paper, let $\k$ be a f\/ield and let $\k^*$ denote the multiplicative group $\k {\setminus} \{0\}$. Fix a positive integer $n \ge 2$. Always assume that there exists in~$\k$ a primitive $n$-th root of unity $\omega$ and $\operatorname{char} \k \nmid n$. By $\Z_n= \{\bar{0},\bar{1},\dots,\overline{n-1}\}$ we denote the cyclic group of order $n$ written additively.

\section{Symmetric linear Gr-categories and algebras therein}\label{section2}

The aim of this section is to collect some preliminary concepts and notations about symmetric linear Gr-categories and various algebraic structures within them.

\subsection{Symmetric linear Gr-categories}\label{section2.1}

Let $G$ be a group with unit $e$ and $\k$ a f\/ield. By $\Vec_G$ we denote the category of f\/inite-dimensional $\k$-spaces graded by the group $G$. Adorned with the usual tensor product of $G$-graded vector spaces and an associativity constraint given by a normalized 3-cocycle $\Phi$ on $G$, i.e., a function $\Phi \colon G \times G \times G \rightarrow \k^*$ such that for all $x,y,z,t \in G$
\begin{gather*}
\Phi(y,z,t)\Phi(x,yz,t)\Phi(x,y,z)=\Phi(x,y,zt)\Phi(xy,z,t) \qquad \mathrm{and} \qquad \Phi(x,e,y)=1,
\end{gather*}
then $\Vec_G$ becomes a tensor category which is called a $\k$-linear Gr-category over $G$ and is denoted by $\Vec_G^\Phi$ in the rest of the paper. When $\Phi$ vanishes, i.e., $\Phi(x,y,z) \equiv 1$, we write simply~$\Vec_G$. A~linear Gr-category $\Vec_G^\Phi$ is said to be braided, if there exists a braiding given by a~quasi-bicharacter $\R$ with respect to~$\Phi$, that is a function $\R\colon G \times G \rightarrow \k^*$ satisfying
\begin{gather*}
\R(xy, z)=\R(x , z)\R(y ,z)\frac{\Phi(z,x,y)\Phi(x,y,z)}{\Phi(x,z,y)},\\
\R(x , yz)=\R(x , y)\R(x , z)\frac{\Phi(y,x,z)}{\Phi(y,z,x)\Phi(x,y,z)}
\end{gather*}
 for all $x,y,z \in G$. Note that a necessary condition for $\Vec_G^\Phi$ being braided is that $G$ is abelian. A~sym\-metric Gr-category is a braided Gr-category $(\Vec_G^\Phi, \R)$ in which the braiding $\mathcal{R}$ is symmetric, in other words, $\mathcal{R}(x,y)\mathcal{R}(y,x)=1$ for any $x,y\in G$. Note that, symmetric braidings of~$\Vec_G$ are nothing other than the familiar symmetric bicharachers of~$G$.

For more details on symmetric linear Gr-categories, the reader is referred to \cite{hly1,hly2}.

\subsection{Algebras and coalgebras in Gr-categories}\label{section2.2}
An algebra in the linear Gr-category $\Vec_G^\Phi$ is an object $A$ with a multiplication morphism $m\colon A \otimes A \to A$ and a unit morphism $u\colon \k \to A$ satisfying the associativity and the unitary conditions of usual associative algebras but expressed in diagrams of the category $\Vec_G^\Phi$. More precisely, an algebra $A$ in $\Vec_G^\Phi$ is a f\/inite-dimensional $G$-graded space $\oplus_{g \in G} A_g$ with a~multiplication~$\cdot$ such that \begin{gather*} A_g \cdot A_h
\subseteq A_{gh}, \qquad (a \cdot b) \cdot c = \Phi(|a|,|b|,|c|) a
\cdot (b \cdot c) \end{gather*} for all homogeneous elements $a,b,c
\in A$. Here and below, $|a|$ denotes the $G$-degree of~$a$. There is also a unit
element~$1$ in $A$ such that
\begin{gather*}
 1 \cdot a=a=a \cdot 1
\end{gather*}
 for all $a \in A$. Further, if the Gr-category $\Vec_G^\Phi$ is braided with braiding $\R$, then we say that two homogeneous elements $a,b \in A$ are commutative if
 \begin{gather*}
 a\cdot b=\R(|a|,|b|)b\cdot a.
 \end{gather*}
 If any two homogeneous elements commute, then we say that~$A$ is a commutative algebra in $\big(\Vec_G^\Phi,\R\big)$.

Dually, a coalgebra in the linear Gr-category $\Vec_G^\Phi$ is an object $C$
with a comultiplication morphism $\Delta\colon C \to C \otimes C$ and a counit
morphism $\varepsilon\colon C \to \k$ satisfying the coassociativity and counitary
axioms in the category $\Vec_G^\Phi$.

We remark that algebras in $\Vec_G^\Phi$ are also called $G$-graded quasialgebras in~\cite{am1}. In addition, morphisms for algebras and coalgebras in $\Vec_G^\Phi$ can be def\/ined in an obvious manner and thus we omit further details.

\subsection{Twisted group algebras}\label{section2.3}

Algebras in $\Vec_G^\Phi$ may be viewed as a natural generalization of the familiar twisted group algebras. Let $G$ be a group and $F$ a normalized 2-cochain on $G$, i.e., a function $F \colon G\times G \rightarrow \k^*$ with $F(e,x)=F(x,e)=1$. Then we can def\/ine a twisted group algebra $\k_F[G]$ which has the same vector space as the group algebra $\k [G]$ but a dif\/ferent product twisted by $F$, namely
\begin{gather*}
 g \star h=F(g,h)gh, \qquad \forall\, g,h\in G.
 \end{gather*}
Clearly, $\k_F[G]$ is an algebra in the linear Gr-category $\Vec_G^{\partial F}$, where
\begin{gather*}
\partial F (x,y,z)=\frac{F(x,y)F(xy,z)}{F(x,yz)F(y,z)}
\end{gather*}
is the dif\/ferential of~$F$. Note that, the twisted group algebra $\k_F[G]$ is associative in the usual sense if and only if~$F$ is a 2-cocycle, i.e., $\partial F=0$, on~$G$. Furthermore, if~$G$ is abelian then $\Vec_G^{\partial F}$ is symmetric with braiding given by
\begin{gather*}\R_F(g,h)=\frac{F(g,h)}{F(h,g)}, \qquad \forall\, g,h \in G,
\end{gather*} and $\k_F[G]$ is commutative in the symmetric Gr-category $\big(\Vec_G^{\partial F}, \R_F\big)$.

Many interesting algebras appear as twisted group algebras. Here we recall some examples presented in~\cite{am1,am2,mgo}. Let $\r$ denote the f\/ield of real numbers, $\Z_2=\{0,1\}$ the cyclic group of order 2, and $\Z_2^n$ the direct product of $n$ copies of $\Z_2$. Elements of $\Z_2^n$ are written as $n$-tuples of $\{0,1\}$ and the group product is written as $+$. Def\/ine functions $f_m\colon \Z_2^n \times \Z_2^n \to \Z_2$ for all $1\le m \le 3$ by
\begin{gather*} f_1(x,y)=\sum_{i} x_iy_i,\qquad f_2(x,y)=\sum_{i<j} x_iy_j, \qquad
  f_3(x,y)=\sum\limits_{\substack{{\rm distinct}\ i,\,j,\,k \\ i<j}} x_ix_jy_k. \end{gather*}
\begin{enumerate}\itemsep=0pt
  \item   Let $F_{\rm Cl}\colon \Z_2^n \times \Z_2^n \to \r^*$ be a function def\/ined by
\begin{gather*} F_{\rm Cl}(x,y)=(-1)^{f_1(x,y)+f_2(x,y)}. \end{gather*}
Then the associated twisted group algebra~$\r_{F_{\rm Cl}} [\Z_2^n]$ is
the well-known real Clif\/ford algebra~${\rm Cl}_{0,n}$, see~\cite{am2} for detail.
This recovers the algebra of complex numbers $\mathbb{C}$ when $n=1$ and the algebra of quaternions $\mathbb{H}$ when $n=2$.
Note that ${\rm Cl}_{0,n}$ is associative in the usual sense since the function $F_{\rm Cl}$ is a 2-cocycle.
  \item Assume $n \ge 3$. Def\/ine the function $F_\mathbb{O}\colon \Z_2^n \times \Z_2^n \to \r^*$ by
\begin{gather*}
F_\mathbb{O}(x,y)=(-1)^{f_1(x,y)+f_2(x,y)+f_3(x,y)}.
\end{gather*}
Then the twisted group algebra $\r_{F_\mathbb{O}} [\Z_2^n]$ is the algebra of higher octonions~$\mathbb{O}_n$
introduced in \cite{mgo} by generalizing the realization of octonions via twisted group algebras
(i.e., the case of $n=3$) observed in~\cite{am1}.
Clearly, the series of algebras~$\mathbb{O}_n$ are nonassociative in the usual sense as the function~$F_\mathbb{O}$ is not a 2-cocycle.
\end{enumerate}

\subsection{Braided tensor products}\label{section2.4}
If $A_1$ and $A_2$ are two algebras in a symmetric linear Gr-category $\big(\Vec_G^\Phi, \R\big)$, then there is a~natural braided tensor product algebra $A_1 \widehat{\otimes} A_2$ living in same category with product given by
\begin{gather*}
\big(a \widehat{\otimes} a'\big) \big(b \widehat{\otimes} b'\big) = \frac{\Phi(|a|,|a'|,|b|)\Phi(|a||b|,|a'|,|b'|)}{\Phi(|a||a'|,|b|,|b'|)\Phi(|a|,|b|,|a'|)}\R(|a'|,|b|)a\cdot b\widehat{\otimes} a'\cdot b',
 \end{gather*}
 for all homogeneous elements $a,b\in A_1$ and $a',b'\in A_2$. Moreover, if $A_1$, $A_2$, $A_3$ are algebras in $\big(\Vec_G^\Phi, \R\big)$, then
 $\big(A_1 \widehat{\otimes} A_2\big)  \widehat{\otimes} A_3$ is naturally isomorphic to $A_1 \widehat{\otimes} \big(A_2 \widehat{\otimes} A_3\big)$. Thus in the following, for brevity we will omit the parentheses for multiple braided tensor product of algebras in symmetric tensor categories.

Dually, if $C_1$ and $C_2$ are two coalgebras in $\big(\Vec_G^\Phi, \R\big)$, then there is a braided tensor product coalgebra $C_1 \widehat{\otimes} C_2$ in the same category with coproduct def\/ined by
\begin{gather*}
\Delta\big(a\widehat{\otimes} b\big) = \frac{\Phi(|a_1|,|a_2|,|b_1|)\Phi(|a_1||b_1|,|a_2|,|b_2|)}{\Phi(|a_1||a_2|,|b_1|,|b_2|)\Phi(|a_1|,|b_1|,|a_2|)}\R(|a_2|,|b_1|)
(a_1\widehat{\otimes} b_1) \otimes (a_2\widehat{\otimes} b_2),
 \end{gather*}
 for all homogeneous $a\in C_1$, $b\in C_2$. Here and in the following we use Sweedler's sigma notation $\Delta(a)=a_1 \otimes a_2$ for coproduct.

\subsection{Weak Hopf algebras in symmetric Gr-categories}\label{section2.5}

With braided tensor product, we can def\/ine Hopf algebras and their various generalizations in symmetric Gr-categories. Of course, one may even def\/ine these algebraic structures in general braided tensor categories, see~\cite{majid1, majid2}.
For simplicity and for our purpose, we only recall the notion of weak Hopf algebras in symmetric Gr-categories $\big(\Vec_G^\Phi, \R\big)$.

\begin{Definition}\label{definition2.1}
Call $(H,m,\mu,\Delta,\varepsilon, S)$ a weak Hopf algebra in $\big(\Vec_G^\Phi, \R\big)$, if
\begin{enumerate}\itemsep=0pt
  \item[(1)] $(H,m,\mu)$ is an algebra in $\big(\Vec_G^\Phi, \R\big)$;
  \item[(2)] $(H,\Delta,\varepsilon)$ is a coalgebra in $\big(\Vec_G^\Phi, \R\big)$;
  \item[(3)] $\Delta\colon H\To H \widehat{\otimes} H$ is multiplicative, i.e., $\Delta(xy)=\Delta(x)\Delta(y)$, $\forall\, x, y \in H$;
  \item[(4)] the following identities hold:
  \begin{gather}
  \varepsilon((fg)h)=\Phi(|fg_1|,|g_2|,|h|)\Phi^{-1}(|f|,|g_1|,|g_2|)\varepsilon(fg_1)\varepsilon(g_2h)\nonumber\\
  \hphantom{\varepsilon((fg)h)}{}
  =\Phi(|fg_2|,|g_1|,|h|)\Phi^{-1}(|f|,|g_2|,|g_1|)\R(|g_1|,|g_2|)\varepsilon(fg_2)\varepsilon(g_1h) \nonumber\\
  \hphantom{\varepsilon((fg)h)=}{}
   \ \forall \ \mathrm{homogeneous \ elements} \ f,g,h \in H,\label{def-eq1}\\
  (1_{11}\otimes 1_{12})\otimes1_{2}= \Phi(|1_1|,|1_2|,|1_{(1)}|)(1_1\otimes 1_{2}1_{(1)})\otimes1_{(2)}\nonumber\\
 \hphantom{(1_{11}\otimes 1_{12})\otimes1_{2}}{}
  = \Phi(|1_1|,|1_2|,|1_{(1)}|)\R(|1_2|,|1_{(1)}|)(1_{1}\otimes 1_{(1)}1_2)\otimes1_{(2)},\label{def-eq2}
  \end{gather}
  where we write $\Delta(1)=1_1 \otimes 1_2=1_{(1)} \otimes 1_{(2)}$;
 \item[(5)] $S \colon H \To H$ is a morphism satisfying
\begin{gather*}
 h_1S(h_2)=\Phi(|1_1|,|1_2|,|h|)\Phi^{-1}(|1_1|,|h|,|1_2|)\R(|1_2|,|h|)\varepsilon(1_1h)1_2,
\\
 S(h_1)h_2=\Phi^{-1}(|h|,|1_1|,|1_2|)\Phi(|1_1|,|h|,|1_2|)\R(|h|,|1_1|)1_1\varepsilon(h1_2),
 \\
 S(h)=(S(h_1)h_2)S(h_3)
\end{gather*}
for all homogeneous elements $h\in H$.
\end{enumerate}
\end{Definition}

Note that a weak Hopf algebra in $\big(\Vec_G^\Phi, \R\big)$ is a Hopf algebra in $\big(\Vec_G^\Phi, \R\big)$ if and only if the comultiplication is unit-preserving, and if and only if the counit is a morphism of algebras in $\big(\Vec_G^\Phi, \R\big)$. Moreover, if $H_1$ and $H_2$ are (weak) Hopf algebras in $\big(\Vec_G^\Phi, \R\big)$, then so is $H_1 \widehat{\otimes} H_2$.

\subsection{Gauge transformations}\label{section2.6}
When considering algebras in symmetric linear Gr-categories, we are entitled to apply the po\-wer\-ful idea of gauge transformations invented for general quasi-Hopf algebras by Drinfeld~\cite{d}. In fact, this point of view is crucial in the pioneering works of Albuquerque and Majid~\cite{am1, am2}. For the purpose of the present paper, it suf\/f\/ices to recall gauge transformations in the following special situation.

Let $G$ be a f\/inite abelian group and $\Phi$ a 3-cocycle on $G$. It is well known that $\k [G]$ has a natural Hopf algebra structure. If we extend $\Phi$ trilinearly to $(\k [G])^{\otimes 3}$, then we have a dual quasi-Hopf algebra $(\k [G], \Phi)$ where $\Phi$ is an associator, see \cite{am1}. In this terminology, the linear Gr-category $\Vec_G^\Phi$ is exactly the comodule category of $(\k [G], \Phi)$. By def\/inition, a gauge transformation, or a twisting, on the dual quasi-Hopf algebra $(\k G, \Phi)$ is a convolution-invertible linear map $F \colon \k [G] \otimes \k [G] \to \k$ satisfying $F(g, e) = F(e, g)=1$ for all $g \in G$. In other words, $F$ is a bilinear expansion of a normalized 2-cochain on $G$. Given a gauge transformation $F$ on $Q=(\k [G], \Phi)$, one can def\/ine a new dual quasi-Hopf algebra $Q_F=(\k [G], \Phi \partial F)$, see \cite{am1} for more details. The corresponding linear Gr-category of $(\k [G], \Phi_F)$ is $\Vec_G^{\Phi \partial F}$, which is tensor equivalent to $\Vec_G^\Phi$. The tensor equivalence is given by $(\mathcal{F}, \varphi_0, \varphi_2)\colon \Vec_{G}^\Phi  \To \Vec_{G}^{\Phi \partial F}$ with
\begin{gather*}
\mathcal{F}(U)=U, \qquad \varphi_0=\Id_\k, \qquad \varphi_2\colon \ U \otimes V \To U \otimes V, \qquad u \otimes v \mapsto F(x,y) u \otimes v
\end{gather*}
for all $U, V$ and $u \in U_x$, $v \in V_y$. The tensor equivalence $\mathcal{F}$ also induces a one-to-one correspondence between the set of braidings of $\Vec_{G}^\Phi$ and that of $\Vec_{G}^{\Phi \partial F}$. Given a braiding $\mathcal{R}$ of $\Vec_{G}^\Phi$, def\/ine $\mathcal{F}(\mathcal{R})$ by \begin{gather*}
\mathcal{F}(\mathcal{R})(x,y)=\frac{F(x,y)}{F(y,x)}\mathcal{R}(x,y), \qquad \forall\, x, y \in G.
\end{gather*}
Then $\mathcal{F}(\mathcal{R})$ is a braiding of $\Vec_{G}^{\Phi \partial F}$ and
\begin{gather*}
 (\mathcal{F}, \varphi_0, \varphi_2)\colon \ \big(\Vec_{G}^\Phi, \mathcal{R}\big)  \To \big(\Vec_{G}^{\Phi \partial F}, \mathcal{F}(\mathcal{R})\big) \end{gather*}
 is a braided tensor equivalence. The reader is referred to \cite{am1, ck1, majid2} for unexplained concepts and notations.

Clearly, if $A$ is an algebra in $\Vec_G^\Phi$, then $A_F:=\mathcal{F}(A)$ is an algebra in $\Vec_G^{\Phi \partial F}$. Note that the multiplication of $A_F$ is given by
 \begin{gather*}
  x \cdot_{_F} y = F(|x|, |y|)x \cdot y
 \end{gather*}
 for all homogeneous $x, y \in A$. Similarly, if $C$ is a coalgebra in
$\Vec_G^\Phi$, then $C_F:=\mathcal{F}(C)$ is a~coalgebra in $\Vec_G^{\Phi \partial F}$ with twisted comultiplication
 \begin{gather*}
 \Delta_F (u)=F(|u_1|,|u_2|)^{-1}u_1\otimes_{F} u_2
 \end{gather*}
 for all homogeneous $u \in C$. Moreover, if $H$ is a (weak) Hopf algebra in $\Vec_G^\Phi$, then $H_F:=\mathcal{F}(H)$ is a~(weak) Hopf algebra in $\Vec_G^{\Phi \partial F}$, etc. The merit of these gauge transformations lies in that, studying an algebra $A$ in $\Vec_G^\Phi$ amounts to studying the algebra $A_F$ in $\Vec_G^{\Phi \partial F}$. By choosing suitable twisting $F$, there might be a chance to gauge transform an algebra $A$ in $\Vec_G^\Phi$ to a~simpler algebra~$A_F$ in~$\Vec_G^{\Phi \partial F}$.

\section{GCAs as algebras in symmetric Gr-categories}\label{section3}
In this section, we will show how to regard generalized Clif\/ford algebras as algebras in some suitable symmetric linear Gr-categories and apply this point of view to derive, in an easy manner, the well known theorem of decomposition for GCAs via gauge transformations.

\subsection{Generalized Clif\/ford algebras}\label{section3.1}
Various notions of GCAs appear in the literature, see for example \cite{j,knus, long,milnor, mo1, mo2, t}. In the present paper, we adopt the following concept which seems to be of the most general form. Assume $q_1, q_2, \dots, q_m \in \k^*$. By def\/inition, the generalized Clif\/ford algebra $C^{(n)}(q_1,q_2,\dots, q_m)$ is the associative $\k$-algebra with unit $1$ generated by $e_1, e_2, \dots, e_m$ subject to the relations
\begin{gather*}
 e_i^n=q_i 1, \qquad \forall\, i \qquad \mathrm{and} \qquad e_ie_j=\omega e_je_i, \qquad \forall\, i>j.
 \end{gather*}

\subsection{Formulas of 2-cocycles}\label{section3.2}
By $\Z_n^m$ we denote the group $\underbrace{\Z_n\times \cdots \times \Z_n}_m$ and by~$[a]$ the integer part of a~rational number~$a$. Let $0 \le x' \le n-1$ denote the residue of the integer~$x$ divided by~$n$. Elements $g \in \Z_n^m$ are written in the form $(g_1,g_2,\dots,g_m)$ and the product of $\Z_n^m$ is written as~$+$.

Take a primitive $n$-th root of unity $\omega$ and f\/ix $q_1,\dots, q_m\in \k^*$. Def\/ine $F \colon \Z_n^m\times \Z_n^m\rightarrow \k^*$ by
\begin{gather} \label{eq3.1}
F_{\rm GCA}(g,h)=\omega^{\sum\limits_{1\leq j<i\leq m}g_ih_j}\prod_{i=1}^m q_i^{[\frac{g_i+h_i}{n}]}
\end{gather}
for all $g,h\in \Z_n^m$. It is easy to see that $F(g,0)=F(0,g)=1$, and that
\begin{gather*}
\partial F_{\rm GCA}(f,g,h)= \frac{F_{\rm GCA}(f,g)F_{\rm GCA}(f+g,h)}{F_{\rm GCA}(g,h)F_{\rm GCA}(f,g+h)}\\
\hphantom{\partial F_{\rm GCA}(f,g,h)}{}
=\prod_{i=1}^mq_i^{[\frac{f_i+g_i}{n}]+[\frac{(f_i+g_i)'+h_i}{n}]-[\frac{g_i+h_i}{n}]-[\frac{f_i+(g_i+h_i)'}{n}]}=1.
\end{gather*}
Hence $F_{\rm GCA}$ is a 2-cocycle on $\Z_n^m$. Def\/ine
\begin{eqnarray}
 \mathcal{R}_{\rm GCA}(g,h)=\frac{F_{\rm GCA}(g,h)}{F_{\rm GCA}(h,g)}=\omega ^{\sum\limits_{1\leq j<i\leq m}(g_ih_j-g_jh_i)}
\end{eqnarray}
and clearly it is a symmetric braiding in $\Vec_{\Z_n^m}$. By Section~\ref{section2.3},
the twisted group algebra $\k_{F_{\rm GCA}}[\Z_n^m]$ is a commutative algebra in the symmetric Gr-category $(\Vec_{\Z_n^m},\R_{\rm GCA})$.
For the convenience of exposition, we write $u_g$ for the element of $\k[\Z_n^m]$ corresponding to~$g\in \Z_n^m$. Then the product on $\k_{F_{\rm GCA}}[\Z_n^m]$
is given by
\begin{gather*}
u_g\star u_h=F_{\rm GCA}(g,h)u_{g+h},
\end{gather*}
where $g,h\in\Z_n^m$ and $g+h=((g_1+h_1)',(g_2+h_2)',\dots,(g_m+h_m)')\in \Z_n^m$.

\subsection{GCAs as twisted group algebras}\label{section3.3}
Keep all the notations of the previous two subsections. Now we have an explicit realization of $C^{(n)}(q_1, q_2,\dots , q_m)$ as a twisted group algebra.

\begin{Proposition} \label{prop3.1}
 $\k_{F_{\rm GCA}}[\Z_n^m]$ and $C^{(n)}(q_1, q_2,\dots, q_m)$ are isomorphic as associative $\k$-algebras. Consequently, the latter is an algebra in the symmetric Gr-category $(\Vec_{\Z^m_n},\R_{\rm GCA})$.
\end{Proposition}

\begin{proof}
It is easy to see that $C^{(n)}(q_1, q_2,\dots, q_m)$ has a basis
$\big\{e_1^{g_1}e_2^{g_2}\cdots e_m^{g_m}\,|\, g_i\in\Z_n,\,i=1,2,\dots,m\big\}$.
Def\/ine a map
\begin{align*}
\psi\colon \ \k_{F_{\rm GCA}}[\Z_n^m]&\To  C^{(n)}(q_1,q_2,\dots ,q_m),&\\
     u_g         &\to   e_1^{g_1}\cdots e_m^{g_m}.&
\end{align*}
Clearly $\psi$ is a linear isomorphism, so it remains to verify that $\psi$ is also multiplicative.
By direct computation, we have
\begin{gather*}
\psi(u_g\star u_h)=\psi(F_{\rm GCA}(g,h)u_{g+h})=F_{\rm GCA}(g,h)\psi(u_{g+h})\\
\hphantom{\psi(u_g\star u_h)}{}
=\omega^{\sum\limits_{1\leq i<j\leq m}g_jh_i}\prod_{i=1}^mq_i^{[\frac{g_i+h_i}{n}]}e_1^{(g_1+h_1)'} \cdots e_m^{(g_m+h_m)'}\\
\hphantom{\psi(u_g\star u_h)}{}
= \big(e_1^{g_1}\cdots e_m^{g_m}\big)\big(e_1^{h_1}\cdots e_m^{h_m}\big) =\psi(u_g)\psi(u_h).
\end{gather*}
Hence we have proved that $\psi$ is an isomorphism of algebras.
\end{proof}

\subsection{Decomposition theorem of GCAs}\label{section3.4}
In this subsection, we apply Proposition~\ref{prop3.1} to derive the well known decomposition theorem of generalized Clif\/ford algebras via gauge transformations.

\begin{Proposition}
We have the following isomorphisms of algebras in $(\Vec_{\Z^m_n},\R_{\rm GCA})$:
\begin{gather*}
C^{(n)}(q_1,\dots ,q_{m}) \cong  C^{(n)}(q_1,\dots ,q_{l})  \widehat{\otimes} C^{(n)}(q_{l+1},\dots ,q_{m})
 \cong C^{(n)}(q_1) \widehat{\otimes}C^{(n)}(q_2)\cdots\widehat{\otimes}C^{(n)}(q_m).
\end{gather*}
\end{Proposition}

\begin{proof}
Recall that $F_{\rm GCA}$ is the function def\/ined in~\eqref{eq3.1} and
\begin{gather*}
\k_{F_{\rm GCA}}[\Z_n^m] \cong C^{(n)}(q_1, q_2,\dots, q_m)
\end{gather*} in $(\Vec_{\Z^m_n},\R_{\rm GCA})$.
Let $F_{\rm GCA}^{-1}$ be the inverse of $F_{\rm GCA}$, namely, $F_{\rm GCA}^{-1}(g,h)=\frac{1}{F_{\rm GCA}(g,h)}$, $\forall\, g, h \in \Z_n^m$. Then $F_{\rm GCA}^{-1}$ is a twisting on the dual quasi-Hopf algebra $(\k_{F_{\rm GCA}}[\Z_n^m], \partial F_{\rm GCA})$ and it induces a braided tensor equivalence: \begin{gather*}
 \big(\mathcal{F}^{-1}, \varphi_0, \varphi_2\big)\colon \  (\Vec_{\Z_n^m}, \mathcal{R} )  \To  (\Vec_{\Z_n^m}, o ),
 \end{gather*}
 where $o(g,h) \equiv 1$. Let $\mathcal{F}$ denote the inverse functor of $\mathcal{F}^{-1}$. By Section~\ref{section2.6}, considering  $\k_{F_{\rm GCA}}[\Z_n^m]$ in $(\Vec_{\Z_n^m}, \mathcal{R})$ amounts to considering $\k[\Z_n^m]$ in $(\Vec_{\Z_n^m}, o)$. Note that $\k[\Z_n^l]$ with $l < m$ can be viewed as an algebra in
$(\Vec_{\Z_n^m}, o)$ by a natural manner,  for example one may assign the element $(g_1, \dots, g_l) \in \Z_n^l$ the grading
$(g_1, \dots, g_l, 0, \dots, 0) \in \Z_n^m$. Thus, by assigning appropriate gradings in this way, one has immediately in
$\left(\Vec_{\Z_n^m}, o\right)$ the following isomorphisms of algebras
\begin{gather} \label{eq3.3}
\k[\Z_n^m] \cong \k\big[\Z_n^l\big] \otimes \k\big[\Z_n^{m-l}\big] \cong \k[\Z_n] \otimes \cdots \otimes \k[\Z_n].
 \end{gather} Note that the braiding $o$ is trivial, hence the braided tensor product in $(\Vec_{\Z_n^m}, o)$ is the usual tensor product and so we omit the ``hats'' in the preceding equation. Now by applying  $\mathcal{F}$ to~\eqref{eq3.3}, the claimed isomorphisms follow.
\end{proof}

\section{Weak Hopf structures of GCAs in symmetric Gr-categories}\label{section4}
In this section, we will show that GCAs have various weak Hopf structures if they are putting in suitable symmetric linear Gr-categories. With a help of the handy gauge transformations, we generalize and simplify the results obtained by Bulacu \cite{b, b1}. More precisely, according to Sections~\ref{section2.6} and \ref{section3.4} the study of $C^n(q_1, \dots, q_m)$ in $(\Vec_{\Z_n^m}, \R_{\rm GCA})$ amounts to that of $\k[\Z_n^m]$ in $(\Vec_{\Z_n^m}, o)$. Apparently, the latter situation seems much simpler. Therefore, we study the weak Hopf structures of $\k[\Z_n^m]$ in $(\Vec_{\Z_n^m}, o)$ f\/irst and then translate to those of $C^n(q_1, \dots, q_m)$ in $(\Vec_{\Z_n^m}, \R_{\rm GCA})$ via gauge transformations.

\subsection[Weak Hopf structures of \protect{$\k[\Z_n]$} in $(\Vec_{\Z_n}, o)$]{Weak Hopf structures of $\boldsymbol{\k[\Z_n]}$ in $\boldsymbol{(\Vec_{\Z_n}, o)}$}\label{section4.1}

For a better exposition of our idea, we start with the simplest case in which $m=1$. In order to f\/ind categorical weak Hopf structures on $\k[\Z_n]$, it suf\/f\/ices to determine its compatible coalgebra structures in $(\Vec_{\Z_n}, o)$. This can be done by direct computation in accordance with the axiom of weak Hopf algebras as follows.

Recall that the elements of $\Z_n$ are denoted by $\bar{i}$, i.e., integers modulo~$n$, and the multiplication is written additively. According to Def\/inition~\ref{definition2.1}(2), the comultiplication must be of the following form
\begin{gather*}
\Delta(\bar{i})=\sum_{\bar{j}} \delta(\bar{j},\overline{i-j}) \bar{j} \otimes \overline{i-j}.
\end{gather*}
The coassociativity of $\Delta$ implies that
\begin{gather*}
\delta(\bar{i}, \overline{j+k})\delta(\bar{j},\bar{k})=\delta(\bar{i},\bar{j})\delta(\overline{i+j}, \bar{k}), \qquad \forall\, \bar{i},\bar{j},\bar{k}.
\end{gather*}
In other words, $\delta$ is a 2-cocycle on $\Z_n$. In addition, the counitary condition says
\begin{gather}\label{eq4.2}
\varepsilon(\bar{0})\delta(\bar{i},\bar{0})=1=\varepsilon(\bar{0})\delta(\bar{0},\bar{i}), \qquad \forall\, \bar{i} \qquad \mathrm{and} \qquad \varepsilon(\bar{j})=0, \qquad \forall\, \bar{j} \ne \bar{0}.
\end{gather}
The multiplicity of $\Delta$ impose on $\delta$ the following
\begin{gather}\label{eq4.3}
\delta(\bar{k}, \overline{i+j-k}) = \sum_{\bar{l}} \delta(\bar{l}, \overline{i-l})\delta(\overline{k-l}, \overline{j+l-k}), \qquad \forall\, \bar{i}, \bar{j}, \bar{k}.
\end{gather}
By~\eqref{def-eq2}, the 2-cocycle $\delta$ satisf\/ies a further condition
\begin{gather}\label{eq4.4}
\delta(\bar{i},\overline{j-i})=\delta(\bar{i},\overline{-i}), \qquad \forall\, \bar{i}, \bar{j}.
\end{gather}
Now it follows immediately from \eqref{eq4.2} and \eqref{eq4.4} that
\begin{gather*}
\delta(\bar{i},\bar{j}) \equiv c, \qquad \forall\, \bar{i},\bar{j}
\end{gather*} where $c \in \k^*$ is a constant. Then by~\eqref{eq4.3}, we have $c=nc^2$, and thus $c=\frac{1}{n}$. With this, the counit is also completely determined, namely,
\begin{gather*}
\varepsilon(\bar{0})=n, \qquad \varepsilon(\bar{i})=0, \qquad \forall\, \bar{i} \ne \bar{0}.
\end{gather*}
It is clear that $\varepsilon$ satisf\/ies~\eqref{def-eq1}. Therefore, endowed with the preceding comultiplication $\Delta$ and counit $\varepsilon$, the algebra $\k[\Z_n]$ becomes a weak bialgebra in the category $(\Vec_{\Z_n}, o)$. Finally, it is not hard to observe that, if we def\/ine a $\k$-linear map $S \colon \k[\Z_n] \To \k[\Z_n]$ by $S(\bar{i})=\bar{i}$, $\forall\, \bar{i}$, then~$S$ satisf\/ies the antipode axiom of weak Hopf algebras (see Def\/inition~\ref{definition2.1}(5)) and~$\k[\Z_n]$ becomes a~weak Hopf algebra in $(\Vec_{\Z_n}, o)$.

To summarize, we have
\begin{Proposition} \label{prop4.1}
 The group algebra $\k[\Z_n]$ has a unique weak Hopf algebra structure in the Gr-category $(\Vec_{\Z_n}, o)$ with comultiplication $\Delta$, counit $\varepsilon$ and antipode $S$ given by
 \begin{gather*}\Delta(\bar{i})=\frac{1}{n}\sum_{\bar{j}} \bar{j} \otimes \overline{i-j}, \qquad \varepsilon(\bar{i})=n\delta_{\bar{i},\bar{0}}, \qquad S(\bar{i})=\bar{i}, \qquad \forall\, \bar{i} \in \Z_{n}
 \end{gather*}
 where $\delta_{\bar{i},\bar{0}}$ is the Kronecker's delta.
\end{Proposition}

\subsection[Weak Hopf structures of \protect{$\k[\Z_n^m]$} in $(\Vec_{\Z_n^m}, o)$]{Weak Hopf structures of $\boldsymbol{\k[\Z_n^m]}$ in $\boldsymbol{(\Vec_{\Z_n^m}, o)}$}\label{section4.2}

Now we extend the above arguments to the general situation. Recall that the elements $g \in \Z_n^m$ are written in the form $(g_1,g_2,\dots,g_m)$ and the product of $\Z_n^m$ is written as~$+$. By a~similar argument, any comultiplication of $\k[\Z_n^m]$ in $(\Vec_{\Z_n^m}, o)$ is determined by a 2-cocycle $\delta \in \operatorname{Z}^2(\Z_n^m, \k^*)$, that is, \begin{gather*} \Delta(u_g)=\sum_{h\in\Z_{n}^m} \delta(h,g-h) u_h \otimes u_{g-h} \end{gather*} and the weak bialgebra axioms in $(\Vec_{\Z_n^m}, o)$ will impose on $\delta$ the following conditions:
\begin{gather}
\delta(f, g+h)\delta(g,h)=\delta(f,g)\delta(f+g,h), \qquad \forall\, f,g,h,\\
\varepsilon(0)\delta(g,0)=1=\varepsilon(0)\delta(0,g),\qquad \forall\, g \qquad \mathrm{and} \qquad \varepsilon(h)=0, \qquad \forall\, h \ne 0,\label{14}\\
\delta(h, f+g-h) = \sum_{l} \delta(l, f-l)\delta(h-l, g+l-h), \qquad \forall\, f,g,h,\\
\delta(g,h-g)=\delta(g,-g),\qquad \forall\, g,h.\label{16}
\end{gather}

As before, it is clear by \eqref{14}--\eqref{16} that
\begin{gather*}
\delta(g,h) \equiv \frac{1}{n^m}, \qquad \forall\, g,h, \qquad \varepsilon(g)=n^m\delta_{g,0}, \qquad \forall\, g.
\end{gather*}
Again, let $S \colon \k[\Z_n^m] \To \k[\Z_n^m]$ be the identity map. It is routine to verify that $(\k[\Z_n^m], \Delta, \varepsilon, S)$ makes a weak Hopf algebra in $(\Vec_{\Z_n^m}, o)$. This provides the main result of this section.

\begin{Theorem} \label{the4.2}
 The group algebra $\k[\Z_n^m]$ has a unique weak Hopf algebra structure in the Gr-category $(\Vec_{\Z_n^m}, o)$ with comultiplication $\Delta$, counit $\varepsilon$ and antipode $S$ given by
 \begin{gather*}\Delta(u_g)=\frac{1}{n^m}\sum_{h} u_h \otimes u_{g-h}, \qquad \varepsilon(u_g)=n^m\delta_{g,0}, \qquad S(g)=g, \qquad \forall\, g\in \Z_{n}^m. \end{gather*}
\end{Theorem}

\subsection{Categorical weak Hopf structures of GCAs}\label{section4.3}
With a help of the results in Section~\ref{section3}, now we can apply the handy gauge transformation to simplify, improve and generalize the results of Bulacu~\cite{b, b1}.

We start with the following
\begin{Corollary} \label{the4.3}
 The twisted group algebra $\k_F[\Z_n^m]$ has a~unique weak Hopf algebra structure in the symmetric Gr-category $(\Vec_{\Z_n^m}^{\partial F}, \R_F)$ with comultiplication~$\Delta$, counit~$\varepsilon$ and antipode~$S$ given~by
 \begin{gather*}
 \Delta(u_g)=\frac{1}{n^m}\sum_{h}F(h,g-h)^{-1} u_{h} \otimes u_{g-h}, \!\qquad \varepsilon(u_g)=n^m\delta_{g,0}, \!\qquad S(g)=g, \!\qquad \forall\, g\in \Z_{n}^m.
 \end{gather*}
\end{Corollary}

\begin{proof}
By applying the gauge transformation
\begin{gather*}
(\mathcal{F}, \varphi_0, \varphi_2)\colon \ (\Vec_{\Z_n^m}, o)  \To \big(\Vec_{\Z_n^m}^{\partial F}, \R_F\big),
 \end{gather*}
one has $\k_F[\Z_n^m]=\mathcal{F}(\k [\Z_n^m])$ and therefore the set of weak Hopf structures of $\k_F[\Z_n^m]$ in $\big(\Vec_{\Z_n^m}^{\partial F}, \R_F\big)$ is in one-to-one correspondence with that of $\k[\Z_n^m]$ in $(\Vec_{\Z_n^m}, o)$. Now the claim follows immediately by the above Theorem~\ref{the4.2}.
\end{proof}

With this, we may discover the categorical weak Hopf structures in suitable symmetric Gr-categories of some well known algebras by properly specifying the twisted functions. In particular, if we take $F=F_{\rm GCA}$, then we determine the categorical weak Hopf structure of generalized Clif\/ford algebras via Proposition \ref{prop3.1}.

\begin{Theorem} \label{the4.4}
 The generalized Clifford algebra $C^n(q_1, \dots, q_m)$ has a unique weak Hopf algebra structure in the symmetric Gr-category $(\Vec_{\Z_n^m}, \R_{F_{\rm GCA}})$, where
  \begin{gather*} F_{\rm GCA}(g,h)=\omega^{\sum\limits_{1\leq j<i\leq m}g_ih_j}\prod_{i=1}^m q_i^{[\frac{g_i+h_i}{n}]},
  \end{gather*}
   with comultiplication $\Delta$, counit $\varepsilon$ and antipode $S$ given by
 \begin{gather*}\Delta\big(e_1^{g_1}e_2^{g_2}\cdots e_m^{g_m}\big)=\frac{1}{n^m}\sum_{h}\omega^{-\sum\limits_{1\leq j<i\leq m}h_i(g_j-h_j)'}\\
 \hphantom{\Delta\big(e_1^{g_1}e_2^{g_2}\cdots e_m^{g_m}\big)=}{}
 \times \prod_{i=1}^m\mathbf{q}_i  e_1^{h_1}e_2^{h_2}\cdots e_m^{h_m} \otimes e_1^{(g_1-h_1)'}e_2^{(g_2-h_2)'}\cdots e_m^{(g_m-h_m)'},
\\
 \varepsilon\big(e_1^{g_1}\cdots e_m^{g_m}\big)=n^m\delta_{g,0}, \qquad \mathrm{and} \qquad S\big(e_1^{g_1}\cdots e_m^{g_m}\big)=e_1^{g_1}\cdots e_m^{g_m},
 \end{gather*}
for all $ e_1^{g_1}e_2^{g_2}\cdots e_m^{g_m}\in C^n(q_1, \dots, q_m)$ and all $g=(g_1,g_2,\dots,g_m)$, $h=(h_1,\dots,h_m)\in \Z_n^m$, where
\begin{gather*}\mathbf{q}_i=
                                                     \begin{cases}
                                                       q_i, &  h_i > g_i, \\
                                                       1, & \text{otherwise}.
                                                     \end{cases}
\end{gather*}
\end{Theorem}

\begin{proof}
Direct consequence of Proposition \ref{prop3.1} and Corollary~\ref{the4.3}.
\end{proof}

\subsection{Final remarks}\label{section4.4}

In fact, our above discussion can be easily extended to the full generality. Namely,
let $G$ be an arbitrary f\/inite abelian group written multiplicatively with unit $e$, let $\k$ be a f\/ield with \mbox{$\operatorname{char} \k\nmid |G|$}, and let $F$ be an arbitrary 2-cochain on $G$. Then the twisted group algebra $\k_F [G]$ has a~unique weak Hopf structure in the symmetric Gr-category $\big(\Vec_{G}^{\partial F}, \R_F\big)$ with comultiplication~$\Delta$, counit $\varepsilon$ and antipode~$S$ given by
 \begin{gather*}\Delta(x)=\frac{1}{|G|}\sum_{u} F\big(u,u^{-1}x\big)^{-1}u \otimes u^{-1}x, \qquad \varepsilon(x)=|G|\delta_{x,e}, \qquad S(x)=x, \qquad \forall\, x\in G. \end{gather*}
If we specify $G=\Z_2^n$ and $F$ the Albuquerque--Majid cochain $F_{\rm CD}$ of Cayley--Dickson algebras~\cite{am1}, then we determine completely the weak Hopf structures of Cayley--Dickson algebras in $\big(\Vec_{\Z_2^n}^{\partial F_{\rm CD}}, \R_{F_{\rm CD}}\big)$. This improves the main result of Bulacu~\cite{b}. If we specify $G=\Z_2^n$ and $F=F_{\mathbb{O}}$, the Morier-Genoud--Ovsienko cochain of higher octonions~\cite{mgo}, then we determine completely the weak Hopf structures of higher octonions~$\mathbb{O}_n$ in the symmetric Gr-category $\big(\Vec_{\Z_2^n}^{\partial F_\mathbb{O}}, \R_{F_\mathbb{O}}\big)$.

\subsection*{Acknowledgements}
 This research  was supported by SRFDP 20130131110001, NSFC 11471186, NSFC 11571199, and SDNSF ZR2013AM022.


\pdfbookmark[1]{References}{ref}
\LastPageEnding

\end{document}